\pdfoutput=1
\documentclass[12pt]{article}
\usepackage{extarrows}
\usepackage{multirow}%
\usepackage{amsmath,amssymb,amsfonts}%
\usepackage{amsthm}%
\usepackage{mathrsfs}%
\usepackage{dsfont}
\usepackage{enumitem}
\usepackage{caption}
\usepackage{graphicx}
\usepackage{subfigure}

\usepackage{tikz}

\allowdisplaybreaks[4]

\usepackage[numbers,sort&compress,comma,square]{natbib}

\usepackage{hyperref}

\hypersetup{
	colorlinks=true,
	anchorcolor=yellow,
	linkcolor=cyan,
	filecolor=blue,
	urlcolor=cyan,
	citecolor=green,
}
\usepackage{cleveref}
\crefformat{equation}{Eq.~(#2#1#3)}

\usepackage[margin=1in]{geometry}

\theoremstyle{definition}
\newtheorem{defn}{Definition}[section]

\theoremstyle{plain}
\newtheorem{thm}[defn]{Theorem}
\newtheorem{lem}[defn]{Lemma}
\newtheorem{cor}[defn]{Corollary}

\newtheorem{prop}[defn]{Proposition}
\theoremstyle{remark}

\newtheorem{ex}{\upshape\bfseries Example}

\newtheoremstyle{case}{}{}{}{}{}{:}{ }{}
\theoremstyle{case}

\DeclareMathOperator{\ii}{\mathrm{i}}
\DeclareMathOperator{\N}{\mathbb{N}}
\DeclareMathOperator{\R}{\mathbb{R}}
\DeclareMathOperator{\Z}{\mathbb{Z}}
\DeclareMathOperator{\Q}{\mathbb{Q}}
\DeclareMathOperator{\D}{\mathbf{D}}
\DeclareMathOperator{\B}{\mathbf{B}}
\DeclareMathOperator{\C}{\mathcal{C}}

\DeclareMathOperator{\IMCG}{\mathrm{IMCG}}
\DeclareMathOperator{\PST}{\mathrm{PST}}
\DeclareMathOperator{\MST}{\mathrm{MST}}
\DeclareMathOperator{\UST}{\mathrm{UST}}

\begin{document}


\title{\textbf{State transfer on integral mixed circulant graphs}\footnote{Supported by the National Natural Science Foundation of China (Grant Nos. 11771141 and 12011530064).}}

\author{Xing-Kun Song\footnote{Email: \href{mailto:xksong@126.com}{xksong@126.com}}, Huiqiu Lin\footnote{Corresponding author. Email: \href{mailto:huiqiulin@126.com}{huiqiulin@126.com} }\\[2mm]
  \small School of Mathematics, East China University of Science and Technology, \\
  \small Shanghai 200237, P.R. China\\}
\date{}
\maketitle

\begin{abstract}
A mixed circulant graph is called integral if all eigenvalues of its Hermitian adjacency matrix are integers. The main purpose of this paper is to investigate the existence of perfect state transfer ($\PST$ for short) and multiple state transfer ($\MST$ for short) on integral mixed circulant graphs. Concretely, we provide sufficient and necessary conditions for the existence of $\PST$ and $\MST$ between specified pairs of vertices on integral mixed circulant graphs, respectively.
\end{abstract}

\begin{flushleft}
	\textbf{Keywords:} Mixed circulant graphs; integral graphs; perfect state transfer; multiple state transfer; Ramanujan's sum.
\end{flushleft}
\textbf{MSC2020:} 05C50; 81P45; 81P68

\section{Introduction}\label{sec::1}

A \textbf{mixed graph} $\Gamma=(V, E, A)$ consists of a set of vertices $V$, a set of undirected edges $E$, and a set of directed edges (or arcs) $A$. In particular,  $\Gamma$ is \textbf{undirected} (resp. \textbf{oriented}) if it contains only undirected (resp. directed) edges. The \textbf{Hermitian adjacency matrix} of $\Gamma$, introduced by Liu and Li~\cite{LL15} and independently by Guo and Mohar~\cite{GM17}, is defined as $H_{\Gamma}=(h_{uv})_{u,v\in V}$, where
\begin{equation*}
	h_{uv} =
	\begin{cases}
		1,    & \text{ if }	\{u,v\}\in E, \\
		\ii,  & \text{ if } (u,v)\in A, \\
		-\ii, & \text{ if } (v,u)\in A, \\
		0,    & \text{ otherwise}.
	\end{cases}
\end{equation*}
Here $\ii=\sqrt{-1}$. Clearly, $H_\Gamma$ is a Hermitian matrix. In particular, if $\Gamma$ is undirected (resp. oriented), then $H_\Gamma=A_\Gamma$ (resp. $H_\Gamma=\ii S_\Gamma$), where $A_\Gamma$ (resp. $S_\Gamma$) is the \textbf{adjacency matrix} (resp. \textbf{skew adjacency matrix}) of $\Gamma$. The eigevalues of $H_\Gamma$ are also called the \textbf{eigenvalues} of $\Gamma$. A mixed graph $\Gamma$ is called  \textbf{integral} if all its eigenvalues are integers.

Let $\Gamma=(V, E, A)$ be a mixed graph. For any vertex $v\in V$, let $\mathbf{e}_v$ be the vector defined on $V$ such that $\mathbf{e}_v(u)=1$ if $u=v$ and $\mathbf{e}_v(u)=0$ otherwise. We say that $\Gamma$ has \textbf{perfect state transfer} ($\PST$ for short) from vertex $u$ to vertex $v$ if there exists a time $t \in \R$ and a complex unimodular scalar $\gamma$ such that
\begin{equation*}
	U(t)\mathbf{e}_u= \gamma \mathbf{e}_v,
\end{equation*}
where $U(t) = \exp(\ii t H_\Gamma)$ is the \textbf{transition matrix} of $H_\Gamma$.
In particular, we say that $\Gamma$ is \textbf{periodic at vertex $u$} if it has perfect state transfer from $u$ to $u$. Furthermore, we say that $\Gamma$ is \textbf{periodic} if $U(t)$ is a scalar multiple of the identity matrix. For any subset $C\subseteq V$, we say that $\Gamma$ has \textbf{multiple state transfer} ($\MST$ for short) on $C$ if $\PST$ occurs between each pair of vertices in $C$.

In \cite{Bo03}, Bose first proposed the concept of $\PST$ for undirected graphs.
In \cite{Go12a}, Godsil proved that for any integer $k$ there are only finitely many (undirected) graphs with maximum degree $k$ on which $\PST$ occurs. Additionally, Godsil \cite{Go08,Go12a,Go12b,CG21b} provided some characterizations for graphs admitting $\PST$. In particular, it was found that  Cayley graphs, especially integral
circulant graphs, play an important role in modeling quantum spin networks supporting the $\PST$ \cite{BP09,BPS09,PB11,SSS07}. In 2010, Angeles-Canul, Norton, Opperman, Paribello, Russell, and Tamon \cite{ANOPRT10} asked for a complete characterization of integral circulant graphs admitting $\PST$. Based on So's characterization~\cite{So06} for integral circulant graphs,  Ba\v{s}i\'{c}~\cite{Ba13} gave an answer to this question. Very recently,  \'{A}rnad\'{o}ttir and Godsil~\cite{AG22} characterized $\PST$ on Cayley graphs for abelian groups that have a cyclic Sylow-$2$-subgroup, which generalizes the main result of Ba\v{s}i\'{c}~\cite{Ba13}. For more results about $\PST$ on (undirected) Cayley graphs, we refer the reader to \cite{CCL20,CFT21,CF21,TFC19,LZ18}, and references therein. With regard to oriented graphs, Godsil and Lato~\cite{GL20}  provided some basic properties for oriented graphs admitting $\PST$, and proposed the concept of $\MST$ for oriented graphs. Song~\cite{S22} gave  necessary and sufficient conditions for the existence of $\PST$ and $\MST$ on integral oriented circulant graphs, respectively.

In this paper, we will consider $\PST$ and $\MST$ for integral mixed circulant graphs. Let $\mathbb{Z}_n$ be the additive group of integers module $n$. Let $\C$ be a subset of $\mathbb{Z}_n\setminus\{0\}$, and let $\overline{\C}=\{c\in \C\mid -c\notin\C\}$.  The \textbf{mixed circulant graph} $G(\mathbb{Z}_n,\C)$ is defined as the mixed graph with vertex set $\mathbb{Z}_n$, edge set $\{\{a,b\}\mid a,b\in \mathbb{Z}_n, b-a\in\C\setminus \overline{\C}\}$ and arc set $\{(a,b)\mid a,b\in \mathbb{Z}_n,b-a\in \overline{\C}\}$. In 2021, Kadyan and Bhattacharjya~\cite{MB21a} characterized all integral mixed circulant graphs.
\begin{thm}{\upshape\bfseries (\cite[Theorem 5.4]{MB21a})}\label{thm::1.1}
Let $G_n(d)=\{k:1\leq k\leq n-1,\gcd(k,n)=d\}$, and let $G_n^r(d)=\{k\in G_n(d): k\equiv r \pmod 4\}$ for $r\in \{1,3\}$. Then the mixed circulant graph $G(\mathbb{Z}_n,\C)$ is integral if and only if
	\begin{equation*}
	\C \setminus \overline{\C} = \bigcup_{d\in \B}G_n(d)
		\text{~~and~~}
		\overline{\C}=
		\begin{cases}
			\emptyset,               & \text{  if } n\not\equiv 0\pmod 4, \\
			\bigcup_{d\in \D}S_n(d), & \text{  if } n\equiv 0\pmod 4,
		\end{cases}
	\end{equation*}
	where $\B \subseteq \{ d: d\mid n,  1\leq d< n\}$, $\D \subseteq \{ d: d\mid n/4, 1\leq d\leq n/4\}$, $\B \cap \D = \emptyset$, and $S_n(d) \in\{ G_n^{1}(d),G_n^{3}(d)\}$.
\end{thm}
Suppose that $G(\mathbb{Z}_n,\C)$ is an integral mixed circulant graph. Let $\B \subseteq \{ d: d\mid n,  1\leq d< n\}$ and $\D \subseteq \{ d: d\mid n/4, 1\leq d\leq n/4\}$ such that $\B \cap \D = \emptyset$. Define $\sigma: \D\rightarrow\{1,-1\}$ be a mapping obtained by $\sigma(d)=1$ if $S_n(d)=G_n^1(d)$, and  $\sigma(d)=-1$ if $S_n(d)=G_n^3(d)$. According to~\Cref{thm::1.1}, if $n\not\equiv 0\pmod 4$, then $\D=\emptyset$ and $G(\mathbb{Z}_n,\C)$ is determined by $\B$, and if $n\equiv 0\pmod 4$, then $G(\mathbb{Z}_n,\C)$ is determined by $\B$, $\D$, and the mapping $\sigma$. For this reason, in what follows, we always use $\IMCG_n(\B,\D,\sigma)$ to denote the integral mixed circulant graph $G(\mathbb{Z}_n,\C)$ with $\C=(\bigcup_{d\in \B}G_n(d))\cup (\bigcup_{d\in \D}G_n^{\sigma(d)}(d))$. Note that here the mapping $\sigma$ exists only when $\D\neq \emptyset$. For example, integral mixed circulant graph $\IMCG_8(\{4\},\{1,2\},\{1,-1\})$ is shown in Fig.~\ref{FIG}. If $G(\mathbb{Z}_n,\C)=\IMCG_8(\{4\},\{1,2\},\{1,-1\})$, then $\C=G_8(4)\cup G_8^{1}(1)\cup G_8^{3}(2)=\{1,4,5,6\}$ with $\overline{\C}=\{1,5,6\}$.

\begin{figure}[htbp!]
  \centering
\begin{tikzpicture}
\foreach \x in{0,1,...,7}
{
\node[circle,fill=gray!25,draw,inner sep=0pt,minimum size=4mm] (a\x) at (-\x*45:2){\x};
\draw[blue] (0,0) -- (a\x);
}
\foreach \x in {0,1,...,7}
{
\pgfmathparse{int(mod(\x+1,8))}
\draw[-stealth,red] (a\x) -- (a\pgfmathresult);
}
\foreach \x in {0,1,...,7}
{
\foreach[parse] \y in{\x+5,\x+6}
{
\pgfmathparse{int(mod(\y,8))}
\draw[-stealth,red] (a\x) -- (a\pgfmathresult);
}
}
\end{tikzpicture}
  \caption{Integral mixed circulant graph $\IMCG_8(\{4\},\{1,2\},\{1,-1\})$ \label{FIG}}
\end{figure}
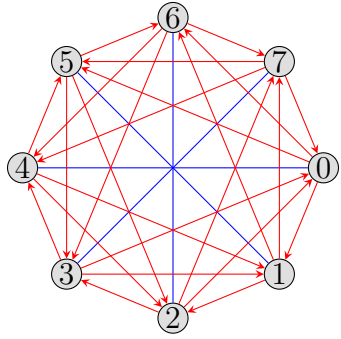

For a positive integer $n$, let $\vartheta_2(n)$ be the largest positive integer $\alpha$ such that $2^{\alpha} \mid n$. For $\B$ and $\D$, we classify and denote by $\B=\bigcup_{i=0}^{\vartheta_2(n)}\B_i$ and $\D=\bigcup_{i=2}^{\vartheta_2(n)}\D_i$, where $\B_i = \{d\in \B \mid \vartheta_2(n/d)=i\}$ and $\D_i = \{d\in \D \mid \vartheta_2(n/d)=i\}$.  The main results of this paper are as follows.

\begin{thm}\label{thm::1.2}
	Let $\Gamma=\IMCG_n(\B,\D,\sigma)$ be an integral mixed circulant graph. Then for all $b \in \Z_n$,  $\Gamma$ has $\PST$ between vertices $b+n/2$ and $b$ if and only if $\B_0=2\B^*_1=4\B^*_2$, and one of the following statements holds:
	\begin{enumerate}[label = \bf(\roman*)]
		\item $n\in 4\N$, $\D_2=\{n/4\}$ and $n/2\notin \B$,\label{thm::1.2::2}
		\item $n\in 4\N$, $\D_2=\emptyset$, either $n/2\in \B$ or $n/4\in \B$, \label{thm::1.2::3}
		\item $n\in 8\N$, $\D_2=\emptyset$, $\{n/4,n/2\}\subseteq \B$, $\D_3=\{n/8\}$, $\B^*_2=2\B^*_3$, \label{thm::1.2::4}
	\end{enumerate}
	where $\B^*_i=\B_{i} \setminus \{n/2^i\}$ for $i\in\{1,2,3\}$.
\end{thm}

\begin{thm}\label{thm::1.3}
	Let $\Gamma=\IMCG_n(\B,\D,\sigma)$ be an integral mixed circulant graph. Then for all $b \in \Z_n$,  $\Gamma$ has $\MST$ between vertices $b$, $b+n/4$, $b+n/2$, $b+3n/4$ if and only if $n\in 8\N$, $\B_0=2\B^*_1=4\B^*_2=8\B^*_3$, $\D_2=\{n/4\}$, $\D_3=\{n/8\}$, $n/2\notin \B$, where $\B^*_i=\B_{i} \setminus \{n/2^i\}$ for $i\in\{1,2,3\}$.
\end{thm}

The paper is organized as follows. In~\Cref{sec::2}, we provide an expression for the eigenvalues of integral mixed circulant graphs by using Ramanujan's sum and Ramanujan’s sine sum. In~\Cref{sec::3}, we give the proof of \Cref{thm::1.2}. In~\Cref{sec::4}, we give the proof of \Cref{thm::1.3}.

\section{The eigenvalues of integral mixed circulant graphs}\label{sec::2}

In this section, we shall express the eigenvalues of integral mixed circulant graphs in terms of Ramanujan's sum, which is crucial in the characterization of PST and MST on such graphs.

Let $n$ be a  positive integer, and let  $\omega_n=\exp(2\pi\ii/n)$ be the primitive $n$-th root of unity. Recall that, for any  positive divisor $d$ of $n$,  $G_n(d)=\{1\leq k\leq n-1: \mathrm{gcd}(k,n)=d\}$, and $G_n^r(d)=\{k\in G_n(d): k\equiv r\pmod 4\}$, where $r=1,3$.  The definitions of Ramanujan's sum and Ramanujan's sine sum are as below.

\begin{defn}{\upshape\bfseries (\cite{R1918})}\label{def::1}
For any positive integers $n$ and $q$, \textbf{Ramanujan's sum} $c_{n}(q)$ is defined as
	\begin{equation*}		c_{n}(q)=\sum_{\substack{1 \leq a \leq n \\ \gcd(a, n)=1}} \omega_n^{aq}=\sum_{a\in G_n(1)} \cos \frac{2 \pi a q}{n}.
	\end{equation*}
\end{defn}
\begin{defn}{\upshape\bfseries (\cite{MB21a}, \cite{S22})}\label{def::2}
	For any positive integers $n$ ($n\equiv 0 \pmod 4$) and $q$, \textbf{Ramanujan's sine sum} $s_{n}(q)$ is defined as
	\begin{displaymath}
		s_{n}(q)=\ii \sum_{a\in G_n^{1}(1)} (\omega^{aq}_n-\omega^{-aq}_n) =-\sum_{a\in G_n^{1}(1)} 2 \sin \frac{2 \pi a q}{n}.
	\end{displaymath}
\end{defn}

Suppose that  $\Gamma=\IMCG_n(\B,\D,\sigma)$ is an integral mixed circulant graph, where $\B \subseteq \{ d: d\mid n,  1\leq d< n\}$, $\D \subseteq \{ d: d\mid n/4, 1\leq d\leq n/4\}$, and  $\sigma$ is a mapping from $ \D$ to $\{\pm 1\}$. Let $\C_1=\bigcup_{d\in \B} G_n(d)$ and $\C_2=\bigcup_{d\in \D} G_n^{\sigma(d)}(d)$.  Since the Hermitian adjacency matrix of $\Gamma$ is a circulant matrix, it is easy to see that $\Gamma$ has the eigenvalues
\begin{equation}\label{eq::001}
\gamma_j=\sum_{k\in \C_1}w_n^{jk}+\sum_{k\in \C_2}\ii(w_n^{jk}-w_n^{jk})=\sum_{d\in \B}\sum_{k\in G_n(d)}\omega_n^{jk}+\sum_{d\in \D}\sum_{k\in G_n^{\sigma(d)}(d)}\ii(w_n^{jk}-w_n^{jk}),
\end{equation}
and the eigenvectors
\begin{equation}\label{eq::002}
	\mathbf{v}_j=[1 \ \omega_n^j \ \omega_n^{2j} \cdots
	\omega_n^{(n-1)j}]^T,
\end{equation}
where $j=0,1,\ldots,n-1$. Observe that   $G_n(d)=dG_{n/d}(1)$ for all $d\in \B$, and if $n\equiv 0\pmod 4$ then $G_n^r(d)=dG_{n/d}^r(1)$ for all $d\in \D$, where $r=\pm 1$. Also note that $\omega_{n/d}=\omega_n^d$. By Definitions \ref{def::1} and \ref{def::2}, the eigenvalues of $\Gamma$ can be expressed as
\begin{equation}\label{eq::003}
\gamma_j=\sum_{d\in\B} c_{n/d}(j)+\sum_{d\in\D} \sigma(d) s_{n/d}(j),
\end{equation}
for $j=0,1,\ldots,n-1$. In what follows, we shall give an equivalent form of \Cref{eq::003}. Before doing this, we need some  basic properties of Ramanujan's sum $c_n(q)$.

\begin{prop}{\upshape\bfseries (\cite{HW08})}\label{prop::2.3}
	Let $\varphi(n)$ be the Euler's totient function. Then we have the following properties:
	\begin{enumerate}[label = \bf(\roman*)]
		\item $c_1(q)=1$, for all positive integers $q$;
        \item if $(m,n) = 1$, then $c_m(q)c_n(q)=c_{mn}(q)$;
		
		\item if $p$ is a prime number, then
		      \[
			      c_{p}(n)=
			      \begin{cases}
				      -1,         & \text{if $p\nmid n$}, \\
				      \varphi(p), & \text{if $p\mid n$};
			      \end{cases}
		      \]
		\item if $p^k$ is a prime power where $k > 1$, then
		      \[
			      c_{p^k}(n) =
			      \begin{cases}
				      0,           & \text{if $p^{k-1}\nmid n$},                       \\
				      -p^{k-1},    & \text{if $p^{k-1}\mid n$ and $p^k\nmid n$}, \\
				      \varphi(p^k)=(p-1)p^{k-1},& \text{if $p^k\mid n$};
			      \end{cases}
		      \]
        \item $c_{n}(q)=\mu(\frac{n}{(n,q)})\frac{\varphi(n)}{\varphi(\frac{n}{(n,q)})}$, where $\mu(n)$ is M\"{o}bius function defined by
        \[
        \mu(n)=
        \begin{cases}
        0,	    &   \text{if $n$ has one or more repeated prime factors},\\
        1,  	&   \text{if $n=1$},\\
        (-1)^k, &   \text{if $n$ is a product of $k$ distinct primes}.
        \end{cases}
        \]
	\end{enumerate}
\end{prop}

For any subset $A\subseteq \mathbb{Q}$, the \textbf{indicator function}  $\mathds{1} _{A}\colon \Q\to \{0,1\}$ is defined by   $$\mathds{1}_{A}(x)=\left\{
\begin{array}{ll}
     1,& \mbox{if $x\in A$,}  \\
     0,& \mbox{otherwise.}
\end{array}\right.
$$

\begin{lem}\label{lem::2.4}
	Let $n=2^{t}m$ where $m$ is a positive odd integer and $t\geq 1$. Then
	\begin{equation}\label{eq::004}
	c_{n}(q)=\mathds{1}_{\Z}(q/2^{t-1}){(-1)}^{q/2^{t-1}}2^{t-1}c_{m}(q').
	\end{equation}
where $q'=q/2^{\vartheta_2(q)}\in 2\N+1$.
\end{lem}

\begin{proof}
Let $m\in 2\N+1$ and $q'=q/2^{\vartheta_2(q)}\in 2\N+1$. Since $(m,q')=(m,q)$ and $(2^{t},m)=1$, by~\Cref{prop::2.3}-{\bf(ii), (v)}, we have
	\begin{equation}\label{eq::005}
		c_{n}(q)=c_{2^{t}m}(q)=c_{2^{t}}(q)c_{m}(q)=c_{2^{t}}(q)c_{m}(q').
	\end{equation}
Furthermore, by~\Cref{prop::2.3}-{\bf (iv)},  $c_{2^{t}}(q)={(-1)}^{q/2^{t-1}}2^{t-1}$ if $q/2^{t-1}\in \mathbb{Z}$, and $c_{2^{t}}(q)=0$ otherwise. Combining this with~\Cref{eq::005}, the result follows.
\end{proof}

For $n\equiv 0\pmod 4$, Song~\cite[Theorem 2.7]{S22} expressed Ramanujan's sine sum $s_{n}(q)$ in terms of Ramanujan's sum.

\begin{lem}{\upshape\bfseries (See~\cite[Theorem 2.7]{S22})}\label{lem::2.5}
	Let $n=2^{t}m$ where  $m$ is an odd positive integer and $t\geqslant 2$. Then
	\[
		s_{n}(q)=
			\mathds{1}_{2\N+1}(q'){(-1)}^{(m-1)/2}{(-1)}^{(q'+1)/2}2^{t-1}c_{m}(q'),
	\]
	where $q'=q/2^{t-2}\in \N$.
\end{lem}

Now we give another expression of the eigenvalues of an integral mixed circulant graph.

\begin{thm}\label{thm::2.6}
	Let $\Gamma=\IMCG_n(\B,\D,\sigma)$ be an integral mixed circulant graph. Then the eigenvalues of $\Gamma$ are
	\begin{equation*}
		\begin{aligned}
			\gamma_j=\sum_{d\in \B_0} c_{n/d}(j/2^{\vartheta_2(j)}) & +\sum_{i=1}^{\vartheta_2(n)}\sum_{d\in \B_i} \mathds{1}_{\Z}(j/2^{i-1}){(-1)}^{j/2^{i-1}} 2^{i-1}c_{\frac{n}{2^{i}d}}(j/2^{\vartheta_2(j)})                                             \\
			                                     & +\sum_{d\in \D_i} \mathds{1}_{2\N+1}(j/2^{i-2}) \sigma(d) {(-1)}^{(\frac{n}{2^{i}d}-1)/2} {(-1)}^{(\frac{j}{2^{i-2}}+1)/2} 2^{i-1} c_{\frac{n}{2^{i}d}} (j/2^{i-2}),
		\end{aligned}
	\end{equation*}
	for $0\leq j\leq n-1$.
\end{thm}

\begin{proof}
  Let $\B=\bigcup_{i=0}^{\vartheta_2(n)} \B_i$ where $\B_i$ is defined in Section \ref{sec::1}. Then
\begin{equation}\label{eq::006}
	\sum_{d\in\B} c_{n/d}(j)=\sum_{\bigcup_{i=0}^{\vartheta_2(n)}\B_i}c_{n/d}(j)=\sum_{d\in \B_0} c_{n/d}(j)+\sum_{i=1}^{\vartheta_2(n)}\sum_{d\in \B_i} c_{n/d}(j).
\end{equation}
Substituting~\Cref{eq::004} into~\Cref{eq::006}, we obtain
\begin{equation}\label{eq::007}
	\sum_{d\in\B} c_{n/d}(j)=\sum_{d\in \B_0} c_{n/d}(j/2^{\vartheta_2(j)})+\sum_{i=1}^{\vartheta_2(n)}\sum_{d\in \B_i} \mathds{1}_{\Z}(j/2^{i-1}){(-1)}^{j/2^{i-1}} 2^{i-1}c_{\frac{n}{2^{i}d}}(j/2^{\vartheta_2(j)}).
\end{equation}
On the other hand, by~\Cref{lem::2.5},
\begin{equation}\label{eq::008}
	\sum_{d\in\D} \sigma(d) s_{n/d}(j)=\sum_{d\in \D_i} \mathds{1}_{2\N+1}(j/2^{i-2}) \sigma(d) {(-1)}^{(\frac{n}{2^{i}d}-1)/2} {(-1)}^{(\frac{j}{2^{i-2}}+1)/2} 2^{i-1} c_{\frac{n}{2^{i}d}} (j/2^{i-2}).
\end{equation}
Combining~\Cref{eq::003},~\Cref{eq::007} and~\Cref{eq::008}, the result follows.
\end{proof}

By Theorem \ref{thm::2.6}, we obtain the following corollary immediately.

\begin{cor}\label{cor::2.7}
	Let $\Gamma=\IMCG_n(\B,\D,\sigma)$ be an integral mixed circulant graph. Then the eigenvalues of $\Gamma$ are $\gamma_j$, for  $0\leq j\leq n-1$, which falls into one of the following three classes:
	\begin{enumerate}[label = \bf(\roman*)]
		\item for $j\in 2\N+1$,
	   	\begin{equation*}
	\gamma_j=\sum_{d\in \B_0} c_{n/d}(j)-\sum_{d\in \B_1} c_{n/(2d)}(j)+2\Lambda_1(j),
	\end{equation*}
	where $\Lambda_1(j)=\sum_{d\in \D_2}\sigma(d) {(-1)}^{(\frac{n}{4d}-1)/2} {(-1)}^{(j+1)/2} c_{\frac{n}{4d}} (j)$;
	\item for $j\in 4\N+2$,
	\begin{equation*}
		\gamma_j=\sum_{d\in \B_0} c_{n/d}(j/2)+\sum_{d\in \B_1} c_{n/(2d)}(j/2)-2\sum_{d\in \B_2} c_{n/(4d)}(j/2)+4 \Lambda_2(j) ,
	\end{equation*}
	 where $\Lambda_2(j)=\sum_{d\in \D_3}\sigma(d) {(-1)}^{(\frac{n}{8d}-1)/2} {(-1)}^{(\frac{j}{2}+1)/2} c_{\frac{n}{8d}} (j/2)$;
	 \item for $j\in 4\N$,
	\begin{equation*}
		\gamma_j=\sum_{d\in \B_0} c_{n/d}(j')+\sum_{d\in \B_1} c_{n/(2d)}(j')+2\sum_{d\in \B_2} c_{n/(4d)}(j')+4\sum_{d\in \B_3} {(-1)}^{j/4}c_{n/(8d)}(j')+8\Lambda_3(j),
	\end{equation*}
	for $j\in 4\N$ and $j'=j/2^{\vartheta_2(j)}$, where
	\begin{equation*}
		\begin{aligned}
			\Lambda_3(j) & =\sum_{i=4}^{\vartheta_2(n)}\sum_{d\in \B_i} \mathds{1}_{\Z}(j/2^{i-1}){(-1)}^{j/2^{i-1}} 2^{i-4}c_{\frac{n}{2^{i}d}}(j')                                                                    \\
			             & +\sum_{\substack{d\in \D_i \\ i\geq  4}} \mathds{1}_{2\N+1}(j/2^{i-2}) \sigma(d) {(-1)}^{(\frac{n}{2^{i}d}-1)/2} {(-1)}^{(\frac{j}{2^{i-2}}+1)/2} 2^{i-4} c_{\frac{n}{2^{i}d}} (j/2^{i-2}).
		\end{aligned}
	\end{equation*}
	\end{enumerate}
\end{cor}

\section{Proof of \texorpdfstring{\Cref{thm::1.2}}{Theorem 1.2}}\label{sec::3}
In the section, we shall prove \Cref{thm::1.2}, which extends some results of~\cite{BPS09,Ba13,S22}.

Let $H$ be the Hermitian adjacency matrix of a mixed graph $\Gamma$ with order $n$. Denote by $\gamma_0, \ldots, \gamma_{n-1}$ the eigenvalues of $H$, and $\mathbf{u}_0,\ldots,\mathbf{u}_{n-1}$  the corresponding orthonormal eigenvectors. By \textbf{spectral decomposition} (see~\cite[Theorem 5.5.1]{Go93}),
$$H=\sum_{r=0}^{n-1}\gamma_{r} \mathbf{u}_{r}\mathbf{u}_{r}^*,$$
where $\mathbf{u}_{r}^*$ denotes the conjugate transpose of $\mathbf{u}_{r}$. Then the transition matrix $U(t)$ of $H$ can be written as
\begin{equation*}
	U(t) =
	\sum_{r=0}^{n-1} \exp(\ii t \gamma_{r}) \mathbf{u}_{r}\mathbf{u}_{r}^*.
\end{equation*}

Suppose that $\Gamma=\IMCG_n(\B,\D,\sigma)$ is an integral mixed circulant graph. Let $H_\Gamma$ be the Hermitian adjacency matrix of $\Gamma$. By
\textbf{spectral decomposition} (see~\cite[Theorem 5.5.1]{Go93}),  the transition matrix $U(t) = \exp(\ii t H_\Gamma)$ of $H_\Gamma$ can be expressed as
\begin{equation}\label{eq::009}
	U(t)=\frac{1}{n} \sum_{j=0}^{n-1} \exp(\ii \gamma_j t) \mathbf{v}_j \mathbf{v}_j^*,
\end{equation}
where $\gamma_j$ and $\mathbf{v}_j$ ($0\leq j\leq n-1$) are the eigenvalues and eigenvectors of $H_\Gamma$ shown in \Cref{eq::001} and \Cref{eq::002}, respectively.
In particular, by \Cref{eq::002} and \cref{eq::009},
\begin{equation}\label{eq::010}
	{U(t)}_{ab}=\frac{1}{n} \sum_{r=0}^{n-1} \exp(\ii \gamma_r t)
	\omega_n^{r(a-b)}=\frac{1}{n} \sum_{r=0}^{n-1} \exp\left(\ii \gamma_r t+\ii\frac{2\pi r (a-b)}{n}\right).
\end{equation}
This expression is given in~\cite[Proposition 1]{SSS07}.
The main goal is to investigate whether there exist distinct integers
$a,b \in \Z_n$ and a positive real number $t$ such that
$\lvert {U(t)}_{ab} \rvert =1$. Let $M_r=\ii (\gamma_r t+2\pi r (a-b)/n)$, for all $1 \leq r\leq n-1$. Obviously, $\lvert {U(t)}_{ab} \rvert \leq 1$, and equality holds if and only if for all $1 \leq r\leq n-1$, the exponents $\exp(M_r)$ are equal in \cref{eq::010}, or equivalently, $(M_{r+1}-M_{r})/(2\pi)\in \Z$. Let $t' = t/(2\pi)$. Then
\[
	\frac{M_{r+1}-M_{r}}{2\pi}=(\gamma_{r+1}-\gamma_r)t'+\frac{a-b}{n}\in \Z,
\]
for all $r=0,\ldots,n-1$. Since $\gamma_r\in \Z$ for all $r=0,\ldots,n-1$, we have $t'\in \Q$. Now we can obtain a necessary and sufficient condition as follows.

\begin{thm}\label{thm::3.1}
	Let $\Gamma=\IMCG_n(\B,\D,\sigma)$ be an integral mixed circulant graph. Then for distinct $a,b \in \Z_n$, $\Gamma$ has $\PST$ between vertices $a$ and $b$
	if and only if there are integers $p$ and $q$ such that
	$\gcd(p,q)=1$ and
	\begin{equation}\label{eq::011}
		\frac{p}{q}(\gamma_{j+1}-\gamma_j)+\frac{a-b}{n}
		\in \Z,
	\end{equation}
	for all $j=0,\ldots,n-1$.
\end{thm}

Suppose that there exists $\gamma_{j}=\gamma_{j+1}$ for $j=0,\ldots,n-1$. By~\Cref{eq::011}, we have $(a-b)/n\in \Z$. Since $0\neq |a-b|<n$, a contradiction. Therefore, we can obtain the following corollary.

\begin{cor}
  Let $\Gamma=\IMCG_n(\B,\D,\sigma)$ be an integral mixed circulant graph. If there exists $\gamma_{j}=\gamma_{j+1}$ for all $j=0,\ldots,n-1$, then $\Gamma$ has no $\PST$.
\end{cor}

According to \Cref{thm::3.1}, if there exists $\PST$ on integral mixed circulant graphs, then we can easily deduce the following corollary.
\begin{cor}\label{cor::4.1}
	Let $\Gamma=\IMCG_n(\B,\D,\sigma)$ be an integral mixed circulant graph. Then for distinct $a,b \in \Z_n$ and $1\leq k\leq n$, $\Gamma$ has $\PST$ between vertices $a$ and $b$
	if and only if there are integers $p$ and $q$ such that
	$\gcd(p,q)=1$ and
	\begin{equation}\label{eq::012}
		\frac{p}{q}(\gamma_{j+k}-\gamma_j)+\frac{k(a-b)}{n}
		\in \Z,
	\end{equation}
	for all $j=0,\ldots,n-1$.
\end{cor}

In general, if an integral mixed circulant graph has $\PST$ between vertices $a$ and $b$, then the order of $a-b$ is two. Furthermore, the order-two element is unique, that is, $a-b=n/2$. By~\Cref{thm::3.1}, we give a necessary and sufficient condition for the existence of $\PST$ on between vertices $b+n/2$ and $b$ on integral mixed circulant graphs. The next lemma is derived from~\Cref{thm::3.1}, which can be used as the criterion for determining the existence of $\PST$. For the proof of~\Cref{lem::3.2}, see~\cite[Lemma 3.3]{S22}.

\begin{lem}{\upshape\bfseries (See~\cite[Lemma 3.3]{S22})}\label{lem::3.2}
	Let $\Gamma=\IMCG_n(\B,\D,\sigma)$ be an integral mixed circulant graph. Then for all $b \in \Z_n$, $\Gamma$ has $\PST$ between vertices $b+n/2$ and $b$ if and only if there exists a
	number $m \in \N$ such that
	\begin{equation} \label{eq:a}
		\vartheta_2(\gamma_{j+1}-\gamma_j)=m,
	\end{equation}
	for all
	$j=0,1,\ldots, n-1$.
\end{lem}

\begin{lem}{\upshape\bfseries (See~\cite[Lemma 3.4]{S22})}\label{lem::3.3}
	Let $n$ be a positive integer and let $\D$ be an odd positive integer set. Then for all odd integer $1\leq q\leq n$ and each positive integers $k$, $\sum_{d\in \D}{(-1)}^{k}c_{d}(q)$ have the same parity if and only if $\D=\{1\}$.

\end{lem}

\begin{lem}\label{lem::3.7}
	Let $\Gamma=\IMCG_n(\B,\D,\sigma)$ be an integral mixed circulant graph. For $j\in 2\N+1$, eigenvalues $\gamma_j$ of $\Gamma$ have the same parity if and only if $\B_0=2(\B_1\setminus\{n/2\})$.
\end{lem}

\begin{proof}
     For $j\in 2\N+1$, by~\Cref{cor::2.7}, the eigenvalues of $\Gamma$ are
	\begin{equation}\label{eq::11}
   \begin{aligned}
   	\gamma_j & =\sum_{d\in \B_0} c_{n/d}(j)-\sum_{d\in \B_1} c_{n/(2d)}(j)+2\Lambda_1(j)                                      \\
   	         & =\sum_{d\in \B_0} c_{n/d}(j)-\sum_{d\in \B_1\setminus\{n/2\}} c_{n/(2d)}(j)-\mathds{1}_{\B}(n/2)+2\Lambda_1(j).
   \end{aligned}
	\end{equation}

  ($\Leftarrow$): If $\B_0=2(\B_1\setminus\{n/2\})$, then we obtain
    	\begin{equation}\label{eq::22}
		\gamma_j=-\mathds{1}_{\B}(n/2)+2\Lambda_1(j).
	\end{equation}
  Therefore, the conclusion ($\Leftarrow$) holds.

  ($\Rightarrow$): (Proof by contradiction) Suppose that  $\B_0\neq 2(\B_1\setminus\{n/2\})$. By~\Cref{lem::3.3}, eigenvalues $\gamma_j$ of $\Gamma$ have no the same parity for $j\in 2\N+1$, therefore, the conclusion holds.

\end{proof}

\begin{lem}\label{lem::3.77}
Let $\Gamma=\IMCG_n(\B,\D,\sigma)$ be an integral mixed circulant graph. If $\Gamma$ has $\PST$. Then all eigenvalues have the same parity. Especially, if $n/2 \in \B$ then all eigenvalues are odd, otherwise they are even.
\end{lem}

\begin{proof}
If $\Gamma$ has $\PST$, by~\Cref{cor::2.7}, then
\begin{equation}\label{eq::19}
		\gamma_j=2\sum_{d\in \B_1\setminus\{n/2\}} c_{n/(2d)}(j/2)+\mathds{1}_{\B}(n/2)-2\sum_{d\in \B_2} c_{n/(4d)}(j/2)+4\Lambda_2(j).
\end{equation}
for all $j\in 4\N+2$.
From~\Cref{eq::22} and \Cref{eq::19}, we can obtain that $\gamma_{j+1}-\gamma_{j}\in 2\N$ for all $j=0,\ldots,n-1$, namely, all eigenvalues have the same parity. Meanwhile, all eigenvalues are odd if $n/2 \in \B$, otherwise even.
\end{proof}

Ba\v{s}i\'{c}~\cite[Theorem 22]{Ba13} and Song~\cite[Theorem 1.2]{S22} gave a characterization of integral circulant graphs and integral oriented circulant graphs admitting PST, respectively.

\begin{lem}{\upshape\bfseries (See~\cite[Theorem 22]{Ba13})}\label{lem::3.91}
Let $\Gamma=\IMCG_n(\B,\emptyset)$ be an integral circulant graph. Then $\Gamma$ has $\PST$ if and only if $n\in 4\N$,
	$\B^*_1=2\B^*_2$, $\B_0=4\B^*_2$ and either $n/4\in \B$ or $n/2\in \B$,
	where $\B^*_2=\B_2\setminus \{n/4\}$ and $\B^*_1=\B_1\setminus \{n/2\}$.
\end{lem}

\begin{lem}{\upshape\bfseries (See~\cite[Theorem 1.2]{S22})}\label{lem::3.10}
	Let $\Gamma=\IMCG_n(\emptyset,\D,\sigma)$ be an integral oriented circulant graph. Then the following three statements are equivalent:
	\begin{enumerate}[label = \bf(\roman*)]
		\item $\Gamma$ has $\PST$ between vertices $b+n/2$ and $b$, for all $b \in \Z_n$;
		\item $n\in 4\N$ and $\D_2=\{n/4\}$.
	\end{enumerate}
\end{lem}

\begin{lem}\label{lem::3.11}
	Let $\Gamma=\IMCG_n(\B,\D,\sigma)$ be an integral mixed circulant graph. If $\Gamma$ has $\PST$, then $n\in 4\N$, $\D_2\setminus\{n/4\}=\emptyset$, $\B_0=2\B^*_1=4\B^*_2$, where $\B^*_2=\B_2\setminus \{n/4\}$ and $\B^*_1=\B_1\setminus \{n/2\}$.
\end{lem}

\begin{proof}
	First of all, we can assume that $\B$ or $\D$ is empty set. By~\Cref{lem::3.91} and \ref{lem::3.10}, the result is obvious. Below we prove that the case of neither $\B$ nor $\D$ is the empty set.

Suppose that $\Gamma$ has PST. By~\Cref{lem::3.7} and~\ref{lem::3.77}, we have $\B_0=2(\B_1\setminus\{n/2\})$. Combining this and~\Cref{cor::2.7} gives that the eigenvalues of $\Gamma$ are
   	\begin{equation*}
   		\begin{aligned}
   			\gamma_j & =\sum_{d\in 2(\B_1\setminus\{n/2\})} c_{n/d}(j/2)+\sum_{d\in \B_1} c_{n/(2d)}(j/2)-2\sum_{d\in \B_2} c_{n/(4d)}(j/2)+4 \Lambda_2(j)                                 \\
   			         & =\sum_{d'\in \B_1\setminus\{n/2\}} c_{n/(2d')}(j/2)+\sum_{d\in \B_1} c_{n/(2d)}(j/2)-2\sum_{d\in \B_2} c_{n/(4d)}(j/2)+4 \Lambda_2(j)                               \\
   			         & =2\sum_{d'\in \B_1\setminus\{n/2\}} c_{n/(2d')}(j/2)+\mathds{1}_{\B}(n/2)-2\sum_{d\in \B_2\setminus \{n/4\}} c_{n/(4d)}(j/2)-2\mathds{1}_{\B}(n/4)+4\Lambda_2(j),
   		\end{aligned}
	\end{equation*}
for $j\in 4\N+2$. Let $\Gamma'=\IMCG_{n/2}(\B',\D',\sigma)$ be an integral mixed circulant graph with $n'=n/2$, $\B'=(\B\setminus\B_0)/2$, and $\D'=(\D\setminus\D_2)/2$. Let $\gamma'_j$ be the eigenvalues of the integral mixed circulant graph $\Gamma'$. Then we have
   	\begin{equation}\label{eq::18}
   		\begin{aligned}
   			\gamma_j &=2\sum_{d'\in \B'_0} c_{n'/d'}(j/2)-2\sum_{d\in \B'_1\setminus \{n'/2\}} c_{n'/(2d)}(j/2)-2\mathds{1}_{\B'}(n'/2)+4\Lambda_1(j/2)+\mathds{1}_{\B}(n/2)\\
   &=2\gamma'_{j/2}+\mathds{1}_{\B}(n/2),
   		\end{aligned}
	\end{equation}
for $j\in 4\N+2$.
Therefore, eigenvalues of $\Gamma$ and $\Gamma'$ are coincide at odd positions. This means that $\gamma'_{j/2}$ have the same parity for $j/2\in 2\N+1$. By~\Cref{lem::3.7}, we have $\B'_0=2(\B'_1\setminus\{n'/2\})$, that is, $\B_1=2(\B_2\setminus\{n/4\})$. Hence, we have
\begin{equation}\label{eq:18}
\gamma_j =\mathds{1}_{\B}(n/2)-2\mathds{1}_{\B}(n/4)+4\Lambda_2(j),
\end{equation}
for $j\in 4\N+2$. According to \Cref{eq::22} and \Cref{eq:18}, we have
\[
\gamma_{j+1}-\gamma_j=2(-\mathds{1}_{\B}(n/4)+2\Lambda_2(j+1)+\mathds{1}_{\B}(n/2)-\Lambda_1(j)).
\]
for $j\in 4\N+1$. By~\Cref{lem::3.2}, $\Lambda_1(j)$ have the same parity for $j\in 4\N+1$. By~\Cref{lem::3.3}, we can obtain that $\D_2\setminus\{n/4\}=\emptyset$.
\end{proof}

For the sake of brevity, let $\Delta(j)=\sum_{d\in \B_2\setminus\{n/4\}} c_{n/(4d)}(j')+\sum_{d\in \B_3} {(-1)}^{j/4}c_{n/(8d)}(j')+2\Lambda_3(j)$ for $j\in 4\N$ and $j'=j/2^{\vartheta_2(j)}$. Similar to the proof of \Cref{lem::3.7}, we can obtain \Cref{lem::3.8}.
\begin{lem}\label{lem::3.8}
For $j\in 4\N$, $\Delta(j)$ have the same parity if and only if $\B^*_2=2\B^*_3$, where $\B^*_2=\B_2\setminus \{n/4\}$ and $\B^*_3=\B_3\setminus \{n/8\}$.
\end{lem}

By~\Cref{lem::3.11}, we re-characterize the eigenvalues of integral mixed circulant graph.

\begin{lem}\label{lem::3.14}
	Let $\Gamma=\IMCG_n(\B,\D,\sigma)$ be an integral mixed circulant graph. If $n\in 4\N$, $\D_2\setminus\{n/4\}=\emptyset$, $\B_0=2\B^*_1=4\B^*_2$, where $\B^*_2=\B_2\setminus \{n/4\}$ and $\B^*_1=\B_1\setminus \{n/2\}$, then the eigenvalues of $\Gamma$ are
\begin{equation*}
    \gamma_j=
    \begin{cases}
    	\mathds{1}_{\B}(n/2)+2\mathds{1}_{\B}(n/4)+4\Delta(j),     & \text{if $j\in 4\N$,}   \\
    	-\mathds{1}_{\B}(n/2)-2\mathds{1}_{\D}(n/4)\sigma(n/4),    & \text{if $j\in 4\N+1$,} \\
    	\mathds{1}_{\B}(n/2)-2\mathds{1}_{\B}(n/4)+4 \Lambda_2(j), & \text{if $j\in 4\N+2$,} \\
    	-\mathds{1}_{\B}(n/2)+2\mathds{1}_{\D}(n/4)\sigma(n/4),    & \text{if $j\in 4\N+3$,}
    \end{cases}
\end{equation*}
where $\Delta(j)=\sum\limits_{d\in \B_2\setminus\{n/4\}} c_{n/(4d)}(j')+\sum\limits_{d\in \B_3} {(-1)}^{j/4}c_{n/(8d)}(j')+2\Lambda_3(j)$ for $j'=j/2^{\vartheta_2(j)}$.
\end{lem}
\begin{proof}
	From the known conditions and~\Cref{cor::2.7}, we can derive some results as follows. For $j\in 4\N+1$, we have
	\begin{equation*}
           		\begin{aligned}
           			\gamma_j & =\sum_{d\in \B_0} c_{n/d}(j)-\sum_{d\in \B_1} c_{n/(2d)}(j)+2\Lambda_1(j)                                      \\
           			         & =\sum_{d\in \B_0} c_{n/d}(j)-\sum_{d\in \B_1\setminus\{n/2\}} c_{n/(2d)}(j)-\mathds{1}_{\B}(n/2)+2\Lambda_1(j) \\
           			         & =-\mathds{1}_{\B}(n/2)-2\mathds{1}_{\D}(n/4)\sigma(n/4).
           		\end{aligned}
	\end{equation*}
Similarly, for $j\in 4\N+3$, we have $\gamma_j=-\mathds{1}_{\B}(n/2)+2\mathds{1}_{\D}(n/4)\sigma(n/4)$.
For $j\in 4\N+2$, we have
	\begin{equation*}
           		\begin{aligned}
           			\gamma_j & =\sum_{d\in \B_0} c_{n/d}(j/2)+\sum_{d\in \B_1} c_{n/(2d)}(j/2)-2\sum_{d\in \B_2} c_{n/(4d)}(j/2)+4 \Lambda_2(j)                                                                            \\
           			         & =\sum_{d\in \B_0} c_{n/d}(j/2)+\sum_{d\in \B_1\setminus\{n/2\}} c_{n/(2d)}(j/2)+\mathds{1}_{\B}(n/2)-2\sum_{d\in \B_2\setminus\{n/4\}} c_{n/(4d)}(j/2)-2\mathds{1}_{\B}(n/4)+4 \Lambda_2(j) \\
           			         & =\mathds{1}_{\B}(n/2)-2\mathds{1}_{\B}(n/4)+4 \Lambda_2(j).
           		\end{aligned}
	\end{equation*}
For $j\in 4\N$, we have
	\begin{equation*}
           		\begin{aligned}
           			\gamma_j & =\sum_{d\in \B_0} c_{n/d}(j')+\sum_{d\in \B_1} c_{n/(2d)}(j')+2\sum_{d\in \B_2} c_{n/(4d)}(j')+4\sum_{d\in \B_3} {(-1)}^{j/4}c_{n/(8d)}(j')+8\Lambda_3(j) \\
           			         & =\mathds{1}_{\B}(n/2)+2\mathds{1}_{\B}(n/4)+4\sum_{d\in \B_2\setminus\{n/4\}} c_{n/(4d)}(j')+4\sum_{d\in \B_3} {(-1)}^{j/4}c_{n/(8d)}(j')+8\Lambda_3(j) \\
           			         & =\mathds{1}_{\B}(n/2)+2\mathds{1}_{\B}(n/4)+4\Delta(j),
           		\end{aligned}
	\end{equation*}
where $\Delta(j)=\sum\limits_{d\in \B_2\setminus\{n/4\}} c_{n/(4d)}(j')+\sum\limits_{d\in \B_3} {(-1)}^{j/4}c_{n/(8d)}(j')+2\Lambda_3(j)$ for $j'=j/2^{\vartheta_2(j)}$.
\end{proof}

\begin{lem}\label{lem::3.13}
	Let $\Gamma=\IMCG_n(\B,\D,\sigma)$ be an integral mixed circulant graph. If $n\in 4\N$, $\D_2\setminus\{n/4\}=\emptyset$, $\B_0=2\B^*_1=4\B^*_2$, where $\B^*_2=\B_2\setminus \{n/4\}$ and $\B^*_1=\B_1\setminus \{n/2\}$, then we have
\begin{equation*}
    \gamma_{j+1}-\gamma_j=
    \begin{cases}
    	-2\mathds{1}_{\B}(n/2)-2\mathds{1}_{\D}(n/4)\sigma(n/4)-2\mathds{1}_{\B}(n/4)-4\Delta(j),       & \text{if $j\in 4\N$,}   \\
    	2\mathds{1}_{\B}(n/2)-2\mathds{1}_{\B}(n/4)+2\mathds{1}_{\D}(n/4)\sigma(n/4)+4 \Lambda_2(j+1)	, & \text{if $j\in 4\N+1$,} \\
    	-2\mathds{1}_{\B}(n/2)+2\mathds{1}_{\D}(n/4)\sigma(n/4)+2\mathds{1}_{\B}(n/4)-4 \Lambda_2(j)	,  & \text{if $j\in 4\N+2$,} \\
    	2\mathds{1}_{\B}(n/2)-2\mathds{1}_{\D}(n/4)\sigma(n/4)+2\mathds{1}_{\B}(n/4)+4\Delta(j+1),      & \text{if $j\in 4\N+3$.}
    \end{cases}
\end{equation*}
where $\Delta(j)=\sum\limits_{d\in \B_2\setminus\{n/4\}} c_{n/(4d)}(j')+\sum\limits_{d\in \B_3} {(-1)}^{j/4}c_{n/(8d)}(j')+2\Lambda_3(j)$ for $j'=j/2^{\vartheta_2(j)}$.
\end{lem}

\begin{proof}[\upshape\bfseries Proof of \Cref{thm::1.2}]
Assume that $n\in 4\N$, $\D_2\setminus\{n/4\}=\emptyset$, $\B_0=2\B^*_1=4\B^*_2$, where $\B^*_2=\B_2\setminus \{n/4\}$ and $\B^*_1=\B_1\setminus \{n/2\}$. We prove that the following three cases.

\begin{description}
\item[Cases 1] $\D_2=\{n/4\}$ and $n/2\notin \B$.

Assume that $\D_2=\{n/4\}$. By~\Cref{lem::3.13}, we have
\begin{equation*}
  \gamma_{j+1}-\gamma_j=
    \begin{cases}
		4\Lambda_2(j+1)+2\mathds{1}_{\B}(n/2)+2\sigma(n/4)	, & \text{if $j\in 4\N+1$,} \\
		-4\Lambda_2(j)-2\mathds{1}_{\B}(n/2)+2\sigma(n/4)	, & \text{if $j\in 4\N+2$.}
    \end{cases}
\end{equation*}
$(\Rightarrow)$ (Proof by contradiction) Suppose that $n/2\in \B$. For $j\in 4\N+1$, we have
\begin{equation}\label{eq::23}
  \gamma_{j+1}-\gamma_j=
    \begin{cases}
    	4(\Lambda_2(j+1)+1)	, & \text{if $\sigma(n/4)=1$,}  \\
    	4\Lambda_2(j+1)	,     & \text{if $\sigma(n/4)=-1$.}
    \end{cases}
\end{equation}
For $j\in 4\N+2$, we have
\begin{equation}\label{eq::24}
  \gamma_{j+1}-\gamma_j=
    \begin{cases}
    	-4\Lambda_2(j)	,     & \text{if $\sigma(n/4)=1$,}  \\
    	-4(\Lambda_2(j)+1)	, & \text{if $\sigma(n/4)=-1$.}
    \end{cases}
\end{equation}

From~\Cref{eq::23} and \Cref{eq::24}, we can see that $\Lambda_2(j)+1$ and $\Lambda_2(j)$ have no the same parity for $j\in 4\N+2$, namely, \Cref{eq:a} does not hold. By~\Cref{lem::3.2}, $\Gamma$ has no $\PST$.

$(\Leftarrow)$ If $n/2\notin \B$, by~\Cref{lem::3.13}, then
\begin{equation*}
  \gamma_{j+1}-\gamma_j=
    \begin{cases}
    	-2\sigma(n/4)-4\Delta(j),     & \text{if $j\in 4\N$,}   \\
    	2\sigma(n/4)+4\Lambda_2(j+1), & \text{if $j\in 4\N+1$,} \\
    	2\sigma(n/4)-4\Lambda_2(j),   & \text{if $j\in 4\N+2$,} \\
    	-2\sigma(n/4)+4\Delta(j+1),   & \text{if $j\in 4\N+3$.}
    \end{cases}
\end{equation*}
Therefore, $\vartheta_2(\gamma_{j+1}-\gamma_j)=1$ for all
	$j=0,1,\ldots, n-1$, namely, $\Gamma$ has $\PST$.
\item[Cases 2] $\D_2=\emptyset$, either $n/2\in \B$ or $n/4\in \B$.

For $\{n/4,n/2\}\subseteq \B$, consider in {\bf Cases 3}.

$(\Rightarrow)$ (Proof by contradiction) Suppose that $n/2\notin \B$ and $n/4\notin \B$. If $\D_2=\emptyset$, by~\Cref{lem::3.13}, then
\begin{equation}\label{eq::25}
    \gamma_{j+1}-\gamma_j=
    \begin{cases}
    	-4\Delta(j),      & \text{if $j\in 4\N$,}   \\
    	4 \Lambda_2(j+1), & \text{if $j\in 4\N+1$,} \\
    	-4 \Lambda_2(j),  & \text{if $j\in 4\N+2$,} \\
    	4\Delta(j+1),     & \text{if $j\in 4\N+3$.}
    \end{cases}
\end{equation}
If $\Gamma$ has $\PST$, by~\Cref{lem::3.2} and~\Cref{eq::25}, then $\Lambda_2(j)$ and $\Delta(j)$ have the same parity for $j\in 4\N+2$ and $j\in 4\N$. By~\Cref{lem::3.3}, we have $\D_3=\{n/8\}$, meanwhile, $\Lambda_2(j)\in 2\N+1$ and $n/8\notin \B$. By~\Cref{lem::3.8}, we have  $\B_2\setminus\{n/4\}=2\B_3\setminus\{n/8\}$. Now we have $\Delta(j)=\sum_{d\in \B_3\setminus\{n/8\}} (1+{(-1)}^{j/4})c_{n/(8d)}(j')+2\Lambda_3(j)\in 2\N$ for $j'=j/2^{\vartheta_2(j)}$, a contradiction.

$(\Leftarrow)$ If $\D_2=\emptyset$ and $n/2\in \B$ and $n/4\notin \B$, by~\Cref{lem::3.13}, then
\begin{equation}\label{eq::26}
    \gamma_{j+1}-\gamma_j=
    \begin{cases}
    	-2-4\Delta(j),      & \text{if $j\in 4\N$,}   \\
    	2+4 \Lambda_2(j+1), & \text{if $j\in 4\N+1$,} \\
    	-2-4 \Lambda_2(j),  & \text{if $j\in 4\N+2$,} \\
    	2+4\Delta(j+1),     & \text{if $j\in 4\N+3$.}
    \end{cases}
\end{equation}
Therefore, $\vartheta_2(\gamma_{j+1}-\gamma_j)=1$ for all
	$j=0,1,\ldots, n-1$, namely, $\Gamma$ has $\PST$. Similarly, if $\D_2=\emptyset$ and $n/2\notin \B$ and $n/4\in \B$, then $\Gamma$ has $\PST$.

\item[Cases 3] $\D_2=\emptyset$, $\D_3=\{n/8\}$, $\{n/4,n/2\}\subseteq \B$, $\B^*_2=2\B^*_3$, where $\B^*_3=\B_3\setminus \{n/8\}$.

Assume that $\D_2=\emptyset$ and $\{n/4,n/2\}\subseteq \B$. By~\Cref{lem::3.13}, we have
\begin{equation}\label{eq::27}
    \gamma_{j+1}-\gamma_j=
    \begin{cases}
    	-4-4\Delta(j),    & \text{if $j\in 4\N$,}   \\
    	4 \Lambda_2(j+1), & \text{if $j\in 4\N+1$,} \\
    	-4 \Lambda_2(j),  & \text{if $j\in 4\N+2$,} \\
    	4+4\Delta(j+1),   & \text{if $j\in 4\N+3$.}
    \end{cases}
\end{equation}
$(\Rightarrow)$ (Proof by contradiction) Suppose that $\D_3\neq\{n/8\}$ or $\B_2\setminus\{n/4\}\neq 2\B_3\setminus\{n/8\}$. Then $\Lambda_2(j)$ or $\Delta(j)$ have no the same parity for $j\in 4\N+2$ or $j\in 4\N$. By~\Cref{lem::3.2}, $\Gamma$ has no $\PST$, a contradiction.

$(\Leftarrow)$ If $\D_3=\{n/8\}$ and $\B_2\setminus\{n/4\}= 2\B_3\setminus\{n/8\}$, then $\Lambda_2(j)\in 2\N+1$ and $\Delta(j)+1\in 2\N+1$. Therefore, $\vartheta_2(\gamma_{j+1}-\gamma_j)=2$ for all
$j=0,1,\ldots, n-1$, namely, $\Gamma$ has $\PST$.
\end{description}

\end{proof}

\begin{figure}[htbp!]
  \centering
  \subfigure[$\IMCG_8(\{1\},\{2\},\{-1\})$]{
		\begin{minipage}[b]{0.3\textwidth}
\centering
\begin{tikzpicture}
\foreach \x in{0,1,...,7}
{
\node[circle,fill=gray!25,draw,inner sep=0pt,minimum size=4mm] (a\x) at (-\x*45:2){\x};
}
\foreach \x in {0,1,...,7}
{
\pgfmathparse{int(mod(\x+6,8))}
\draw[-stealth,red] (a\x) -- (a\pgfmathresult);
}
\foreach \x in {0,1,...,7}
{
\foreach[parse] \y in{\x+1,\x+3,\x+5,\x+7}
{
\pgfmathparse{int(mod(\y,8))}
\draw[blue] (a\x) -- (a\pgfmathresult);
}
}
\end{tikzpicture}
\end{minipage}
		\label{fig::2.1}
}
    \subfigure[$\IMCG_8(\{2\},\{1\},\{1\})$]{
    	\begin{minipage}[b]{0.3\textwidth}
    \centering
\begin{tikzpicture}
    \foreach \x in{0,1,...,7}
{
\node[circle,fill=gray!25,draw,inner sep=0pt,minimum size=4mm] (a\x) at (-\x*45:2){\x};
}
\foreach \x in {0,1,...,7}
{
\pgfmathparse{int(mod(\x+1,8))}
\draw[-stealth,red] (a\x) -- (a\pgfmathresult);
}
\foreach \x in {0,1,...,7}
{
\pgfmathparse{int(mod(\x+5,8))}
\draw[-stealth,red] (a\x) -- (a\pgfmathresult);
}
\foreach \x in {0,1,...,7}
{
\foreach[parse] \y in{\x+2,\x+6}
{
\pgfmathparse{int(mod(\y,8))}
\draw[blue] (a\x) -- (a\pgfmathresult);
}
}
\end{tikzpicture}
\end{minipage}
		\label{fig::2.2}
}
   \subfigure[$\IMCG_8(\{2,4\},\{1\},\{-1\})$]{
    	\begin{minipage}[b]{0.3\textwidth}
    \centering
\begin{tikzpicture}
    \foreach \x in{0,1,...,7}
{
\node[circle,fill=gray!25,draw,inner sep=0pt,minimum size=4mm] (a\x) at (-\x*45:2){\x};
}
\foreach \x in {0,1,...,7}
{
\pgfmathparse{int(mod(\x+3,8))}
\draw[-stealth,red] (a\x) -- (a\pgfmathresult);
}
\foreach \x in {0,1,...,7}
{
\pgfmathparse{int(mod(\x+7,8))}
\draw[-stealth,red] (a\x) -- (a\pgfmathresult);
}
\foreach \x in {0,1,...,7}
{
\foreach[parse] \y in{\x+2,\x+4,\x+6}
{
\pgfmathparse{int(mod(\y,8))}
\draw[blue] (a\x) -- (a\pgfmathresult);
}
}

\end{tikzpicture}
      	\end{minipage}
	\label{fig::2.3}
    }
	\caption{Three integral mixed circulant graph having $\PST$}
	\label{fig::2}
\end{figure}
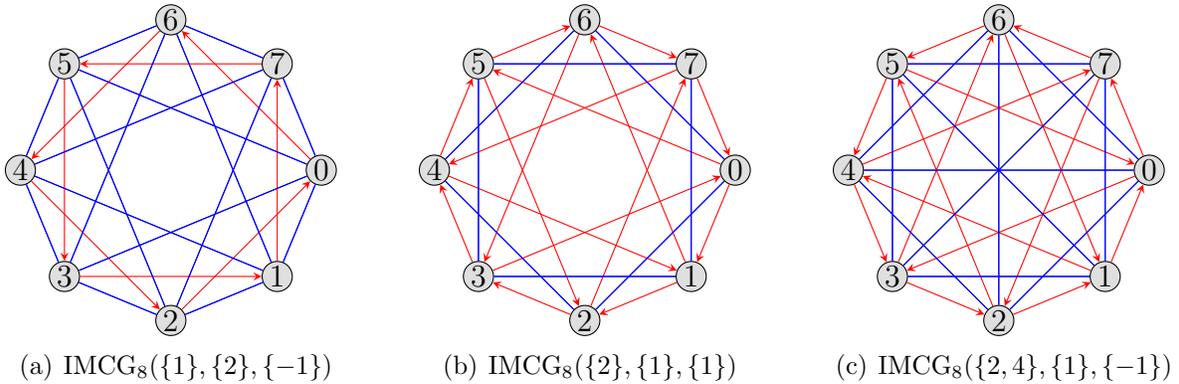

\begin{ex}
We give three examples of the existence of PST for each of the three cases in~\Cref{thm::1.2}, and shown in~Fig.~\ref{fig::2}. Let $\Gamma=\IMCG_n(\B,\D,\sigma)$ be an integral mixed circulant graph.
	\begin{enumerate}[label = \bf(\roman*)]
		\item If $G(\mathbb{Z}_8,\C)=\IMCG_8(\{1\},\{2\},\{-1\})$, then $\C=G_8(1)\cup G_8^{3}(2)=\{1,3,5,6,7\}$ with $\overline{\C}=\{6\}$; (see~Fig.~\ref{fig::2.1})
        \item If $G(\mathbb{Z}_8,\C)=\IMCG_8(\{2\},\{1\},\{1\})$, then $\C=G_8(2)\cup G_8^{1}(1)=\{1,2,5,6\}$ with $\overline{\C}=\{1,5\}$; (see~Fig.~\ref{fig::2.2})
        \item If $G(\mathbb{Z}_8,\C)=\IMCG_8(\{2,4\},\{1\},\{-1\})$, then $\C=G_8(2)\cup G_8(4)\cup G_8^{3}(1)=\{2,3,4,6,7\}$ with $\overline{\C}=\{3,7\}$. (see~Fig.~\ref{fig::2.3})
    \end{enumerate}
\end{ex}

\section{Proof of \texorpdfstring{\Cref{thm::1.3}}{Theorem 1.3}}\label{sec::4}

In~\cite{S22}, Song proved that there exists $\MST$ on integral oriented circulant graphs between vertices $b$, $b+n/4$, $b+n/2$, $b+3n/4$ for all $b \in \Z_n$. We can extend some lemmas to integral mixed circulant graph. By~\Cref{thm::1.2} and~\Cref{lem::3.14}, we have $\mu_{j+4}-\mu_{j}=0$, for each $j$ is odd. According to~\cite{S22}, we can obtain that the following lemma.
\begin{lem}{\upshape\bfseries (See~\cite[Theorem 4.1]{S22})}
	Let $\Gamma=\IMCG_n(\B,\D,\sigma)$ be an integral mixed circulant graph. For any two distinct vertices $a,b\in \mathbb{Z}_n$, if $\Gamma$ has $\PST$ between $a$ and $b$, then $a=b+kn/4$ for some $k\in\{1,2,3\}$.
\end{lem}



\begin{lem}{\upshape\bfseries (See~\cite[Theorem 4.2]{S22})}\label{thm::4.2}
	Let $\Gamma=\IMCG_n(\B,\D,\sigma)$ be an integral mixed circulant graph. Then for all $b \in \Z_n$, $\Gamma$ has $\PST$ between vertices $b+n/4$ and $b$ and between vertices $b+n/2$ and $b$ if and only if $\Gamma$ has $\MST$ between vertices $b$, $b+n/4$, $b+n/2$, $b+3n/4$.
\end{lem}

Suppose that $\Gamma$ has $\PST$ between vertices $b+n/2$ and $b$. By~\Cref{thm::4.2}, if we want to determine the existence of $\MST$ in $\Gamma$, then we only need to satisfy that $\Gamma$ has $\PST$ between vertices $b+n/4$ and $b$.

\begin{lem}{\upshape\bfseries (See~\cite[Theorem 4.3]{S22})}\label{thm::4.4}
	Let $\Gamma=\IMCG_n(\B,\D,\sigma)$ be an integral mixed circulant graph. Suppose that for all $b \in \Z_n$, $\Gamma$ has $\PST$ between vertices $b+n/2$ and $b$. Then for all $b \in \Z_n$, $\Gamma$ has $\PST$ between vertices $b+n/4$ and $b$ if and only if there exists a
	number $m' \in \N$ such that
	\begin{equation}\label{eq:b}
		\vartheta_2(\mu_{j+2}-\mu_j)=m',
	\end{equation}
	for all $j=0,1,\ldots, n-1$.
\end{lem}

If there exists $\gamma_{j}=\gamma_{j+2}$ for all $j=0,\ldots,n-1$, by~\Cref{eq::002}, then $2(a-b)/n\in \Z$. If $a=b+n/4$, then $1/2\in \Z$. This implies that $\Gamma$ has no $\PST$ between vertices $b+n/4$ and $b$.

\begin{cor}\label{cor::4.2}
  Let $\Gamma=\IMCG_n(\B,\D,\sigma)$ be an integral mixed circulant graph. If there exists $\gamma_{j}=\gamma_{j+2}$ for all $j=0,\ldots,n-1$, then $\Gamma$ has no $\PST$ between vertices $b+n/4$ and $b$, therefore has no $\MST$.
\end{cor}

\begin{lem}{\upshape\bfseries (See~\cite[Theorem 4.4]{S22})}\label{thm::4.3}
	Let $\Gamma=\IMCG_n(\B,\D,\sigma)$ be an integral mixed circulant graph. Then for all $b \in \Z_n$, $\Gamma$ has $\MST$ between vertices $b$, $b+n/4$, $b+n/2$, $b+3n/4$ if and only if
	\begin{equation*}
		\vartheta_2(\mu_{j+1}-\mu_j)=1 \text{~and~} \vartheta_2(\mu_{j+2}-\mu_j)=2,
	\end{equation*}
	for all $j=0,1,\ldots, n-1$.
\end{lem}

\begin{lem}\label{lem::3.15}
	Let $\Gamma=\IMCG_n(\B,\D,\sigma)$ be an integral mixed circulant graph. If $n\in 4\N$, $\D_2\setminus\{n/4\}=\emptyset$, $\B_0=2\B^*_1=4\B^*_2$, where $\B^*_2=\B_2\setminus \{n/4\}$ and $\B^*_1=\B_1\setminus \{n/2\}$, then we have
\begin{equation}\label{eq::28}
    \gamma_{j+2}-\gamma_j=
    \begin{cases}
    	-4\mathds{1}_{\B}(n/4)+4 \Lambda_2(j+2)-4\Delta(j), & \text{if $j\in 4\N$,}   \\
    	4\mathds{1}_{\D}(n/4)\sigma(n/4),                   & \text{if $j\in 4\N+1$,} \\
    	4\mathds{1}_{\B}(n/4)+4 \Lambda_2(j)-4\Delta(j+2),  & \text{if $j\in 4\N+2$,} \\
    	-4\mathds{1}_{\D}(n/4)\sigma(n/4),                  & \text{if $j\in 4\N+3$.}
    \end{cases}
\end{equation}
where $\Delta(j)=\sum\limits_{d\in \B_2\setminus\{n/4\}} c_{n/(4d)}(j')+\sum\limits_{d\in \B_3} {(-1)}^{j/4}c_{n/(8d)}(j')+2\Lambda_3(j)$ for $j'=j/2^{\vartheta_2(j)}$.
\end{lem}

\begin{proof}[\upshape\bfseries Proof of \Cref{thm::1.3}]
$(\Rightarrow)$ Assume that $\Gamma$ has $\MST$ between vertices $b$, $b+n/4$, $b+n/2$, $b+3n/4$ for all $b \in \Z_n$. According to~\Cref{lem::3.11},~\ref{thm::4.2},~\ref{lem::3.15}, and~\Cref{cor::4.2}, we have $\mathds{1}_{\D}(n/4)\sigma(n/4)\neq 0$. This implies $\D_2=\{n/4\}$. By~\Cref{thm::1.2}, we have $n/2\notin \B$. Hence, \Cref{eq::28} equal to

\begin{equation}\label{eq::29}
    \gamma_{j+2}-\gamma_j=
    \begin{cases}
    	4 \Lambda_2(j+2)-4\Delta(j), & \text{if $j\in 4\N$,}   \\
    	4\sigma(n/4),                & \text{if $j\in 4\N+1$,} \\
    	4 \Lambda_2(j)-4\Delta(j+2), & \text{if $j\in 4\N+2$,} \\
    	-4\sigma(n/4),               & \text{if $j\in 4\N+3$.}
    \end{cases}
\end{equation}

By~\Cref{thm::4.4} and \Cref{eq::29}, we have $\Lambda_2(j+2)-\Delta(j)\in 2\N+1$ for $j\in 4\N$. By~\Cref{lem::3.3} and~\Cref{lem::3.8}, we can easily obtain that  $\D_3=\{n/8\}$ and  $\B_2\setminus\{n/4\}=2\B_3\setminus\{n/8\}$.

$(\Leftarrow)$ According to $\B_0=2\B^*_1=4\B^*_2$, where $\B^*_2=\B_2\setminus \{n/4\}$ and $\B^*_1=\B_1\setminus \{n/2\}$, $\D_2=\{n/4\}$, and $n/2\notin \B$, by~\Cref{thm::1.2}, we obtain that $\Gamma$ has $\PST$ between vertices $b+n/2$ and $b$ for all $b \in \Z_n$. Since $\D_3=\{n/8\}$ and $\B_2\setminus\{n/4\}=2\B_3\setminus\{n/8\}$, we can easily obtain that $\Lambda_2(j+2)-4\Delta(j) \in 2\N+1$ for $j\in 4\N$ and $\Lambda_2(j)-4\Delta(j+2) \in 2\N+1$ for $j\in 4\N+2$. By~\Cref{eq::29}, we have $\vartheta_2(\gamma_{j+2}-\gamma_j)=2$, for all $j=0,1,\ldots, n-1$. By~\Cref{thm::4.4} and~\ref{thm::4.2}, we can obtain that $\Gamma$ has $\MST$ between vertices $b$, $b+n/4$, $b+n/2$, $b+3n/4$ for all $b \in \Z_n$.
\end{proof}

From~\Cref{thm::1.3}, we can see that $\MST$ on $\Gamma$ only occurs between four vertices. Therefore, we obtain the following corollary.

\begin{cor}
	Let $\Gamma=\IMCG_n(\B,\D,\sigma)$ be an integral mixed circulant graph. Then $\Gamma$ has no $\UST$.
\end{cor}

\begin{ex}
We give an example of the existence of $\MST$ in~\Cref{thm::1.3}, and shown in Fig.~\ref{fig::3}. Let $\Gamma=\IMCG_n(\B,\D,\sigma)$ be an integral mixed circulant graph. If $G(\mathbb{Z}_{16},\C)=\IMCG_{16}(\{1\},\{2,4\},\{1,-1\})$, then $\C=G_{16}(1)\cup G_{16}^{1}(2)\cup G_{16}^{3}(4)=\{1,2,3,5,7,9,10,11,12,13,15\}$ with $\overline{\C}=\{2,10,12\}$.
\end{ex}

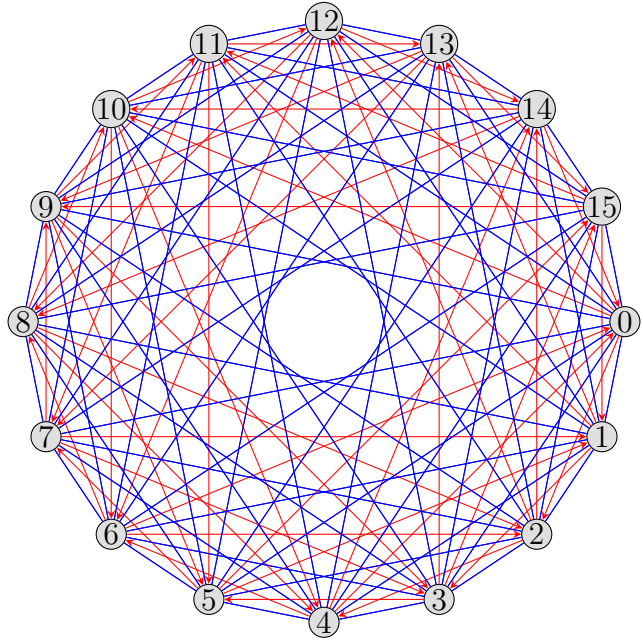
\begin{figure}[htbp!]
  \centering
\begin{tikzpicture}
\foreach \x in{0,1,...,15}
{
\node[circle,fill=gray!25,draw,inner sep=0pt,minimum size=4mm] (a\x) at (-\x*22.5:4){\x};
}
\foreach \x in {0,1,...,15}
{
\pgfmathparse{int(mod(\x+2,16))}
\draw[-stealth,red] (a\x) -- (a\pgfmathresult);
}
\foreach \x in {0,1,...,15}
{
\pgfmathparse{int(mod(\x+10,16))}
\draw[-stealth,red] (a\x) -- (a\pgfmathresult);
}
\foreach \x in {0,1,...,15}
{
\pgfmathparse{int(mod(\x+12,16))}
\draw[-stealth,red] (a\x) -- (a\pgfmathresult);
}
\foreach \x in {0,1,...,15}
{
\foreach[parse] \y in{\x+1,\x+3,\x+5,\x+7,\x+9,\x+11,\x+13,\x+15}
{
\pgfmathparse{int(mod(\y,16))}
\draw[blue] (a\x) -- (a\pgfmathresult);
}
}
\end{tikzpicture}
  \caption{Integral mixed circulant graph $\IMCG_{16}(\{1\},\{2,4\},\{1,-1\})$ having $\MST$ \label{fig::3}}
\end{figure}

\section{Conclusion}
This work focuses on the study of existence for $\PST$ and $\MST$ on integral mixed circulant graphs using Hermitian adjacency matrix. We give a characterization the existence of $\PST$ and $\MST$ on integral mixed circulant graphs. We can find that properly adding oriented edges can help construct graphs with $\PST$ in undirected graphs. For example, integral mixed circulant graphs $G(\mathbb{Z}_{8},\C)=\IMCG_8(\{2,4\},\{1\},\{-1\})$ shown in~\Cref{fig::2.3}, we can see that complete graphs $K_{8}$ have no $\PST$, but $\PST$ is generated when some edges are transformed into directed edges. It can be seen that the mapping $\sigma$ has no effect on the result from the proof of~\Cref{thm::1.2} and \ref{thm::1.3}. At this point, a general question that we might ask is what does the directed edge affect $\PST$ and $\MST$ for a graph using Hermitian adjacency matrix?



\begin{thebibliography}{26}
\expandafter\ifx\csname natexlab\endcsname\relax\def\natexlab#1{#1}\fi
\providecommand{\url}[1]{\texttt{#1}}
\providecommand{\href}[2]{#2}
\providecommand{\path}[1]{#1}
\providecommand{\DOIprefix}{doi:}
\providecommand{\ArXivprefix}{arXiv:}
\providecommand{\URLprefix}{URL: }
\providecommand{\Pubmedprefix}{pmid:}
\providecommand{\doi}[1]{\href{http://dx.doi.org/#1}{\path{#1}}}
\providecommand{\Pubmed}[1]{\href{pmid:#1}{\path{#1}}}
\providecommand{\bibinfo}[2]{#2}
\ifx\xfnm\relax \def\xfnm[#1]{\unskip,\space#1}\fi
\bibitem[{Angeles-Canul et~al.(2010)Angeles-Canul, Norton, Opperman, Paribello,
  Russell, and Tamon}]{ANOPRT10}
\bibinfo{author}{R.~J. Angeles-Canul}, \bibinfo{author}{R.~M. Norton},
  \bibinfo{author}{M.~C. Opperman}, \bibinfo{author}{C.~C. Paribello},
  \bibinfo{author}{M.~C. Russell}, \bibinfo{author}{C.~Tamon},
\newblock \bibinfo{title}{Perfect state transfer, integral circulants, and join
  of graphs},
\newblock \bibinfo{journal}{\emph{Quantum Inf. Comput.}}   \textbf{10}~(3-4)
  (\bibinfo{year}{2010}) \bibinfo{pages}{325--342}.
  \DOIprefix\doi{10.26421/QIC10.3-4-10}.
\bibitem[{\'{A}rnad\'{o}ttir and Godsil(2022)}]{AG22}
\bibinfo{author}{A.~S. \'{A}rnad\'{o}ttir}, \bibinfo{author}{C.~Godsil},
  \bibinfo{title}{On state transfer in cayley graphs for abelian groups},
  \bibinfo{year}{2022}. \href{http://arxiv.org/abs/2204.09802}{{\texttt
  arXiv:2204.09802}}.
\bibitem[{Ba\v{s}i\'{c}(2013)}]{Ba13}
\bibinfo{author}{M.~Ba\v{s}i\'{c}},
\newblock \bibinfo{title}{Characterization of quantum circulant networks having
  perfect state transfer},
\newblock \bibinfo{journal}{\emph{Quantum Inf. Process.}}   \textbf{12}~(1)
  (\bibinfo{year}{2013}) \bibinfo{pages}{345--364}.
  \DOIprefix\doi{10.1007/s11128-012-0381-z}.
\bibitem[{Ba\v{s}i\'{c} and Petkovi\'{c}(2009)}]{BP09}
\bibinfo{author}{M.~Ba\v{s}i\'{c}}, \bibinfo{author}{M.~D. Petkovi\'{c}},
\newblock \bibinfo{title}{Some classes of integral circulant graphs either
  allowing or not allowing perfect state transfer},
\newblock \bibinfo{journal}{\emph{Appl. Math. Lett.}}   \textbf{22}~(10)
  (\bibinfo{year}{2009}) \bibinfo{pages}{1609--1615}.
  \DOIprefix\doi{10.1016/j.aml.2009.04.007}.
\bibitem[{Ba\v{s}i\'{c} et~al.(2009)Ba\v{s}i\'{c}, Petkovi\'{c}, and
  Stevanovi\'{c}}]{BPS09}
\bibinfo{author}{M.~Ba\v{s}i\'{c}}, \bibinfo{author}{M.~D. Petkovi\'{c}},
  \bibinfo{author}{D.~Stevanovi\'{c}},
\newblock \bibinfo{title}{Perfect state transfer in integral circulant graphs},
\newblock \bibinfo{journal}{\emph{Appl. Math. Lett.}}   \textbf{22}~(7)
  (\bibinfo{year}{2009}) \bibinfo{pages}{1117--1121}.
  \DOIprefix\doi{10.1016/j.aml.2008.11.005}.
\bibitem[{Bose(2003)}]{Bo03}
\bibinfo{author}{S.~Bose},
\newblock \bibinfo{title}{Quantum communication through an unmodulated spin
  chain},
\newblock \bibinfo{journal}{\emph{Phys. Rev. Lett.}}   \textbf{91}
  (\bibinfo{year}{2003}) \bibinfo{pages}{207901}.
  \DOIprefix\doi{10.1103/PhysRevLett.91.207901}.
\bibitem[{Cao et~al.(2020)Cao, Chen, and Ling}]{CCL20}
\bibinfo{author}{X.~Cao}, \bibinfo{author}{B.~Chen}, \bibinfo{author}{S.~Ling},
\newblock \bibinfo{title}{Perfect state transfer on {C}ayley graphs over
  dihedral groups: the non-normal case},
\newblock \bibinfo{journal}{\emph{Electron. J. Combin.}}   \textbf{27}~(2)
  (\bibinfo{year}{2020}) \bibinfo{pages}{\#P2.28}.
  \DOIprefix\doi{10.37236/9184}.
\bibitem[{Cao and Feng(2021)}]{CF21}
\bibinfo{author}{X.~Cao}, \bibinfo{author}{K.~Feng},
\newblock \bibinfo{title}{Perfect state transfer on {C}ayley graphs over
  dihedral groups},
\newblock \bibinfo{journal}{\emph{Linear Multilinear Algebra}}
  \textbf{69}~(2) (\bibinfo{year}{2021}) \bibinfo{pages}{343--360}.
  \DOIprefix\doi{10.1080/03081087.2019.1599805}.
\bibitem[{Cao et~al.(2021)Cao, Feng, and Tan}]{CFT21}
\bibinfo{author}{X.~Cao}, \bibinfo{author}{K.~Feng}, \bibinfo{author}{Y.-Y.
  Tan},
\newblock \bibinfo{title}{Perfect state transfer on weighted abelian {C}ayley
  graphs},
\newblock \bibinfo{journal}{\emph{Chinese Ann. Math. Ser. B}}   \textbf{42}~(4)
  (\bibinfo{year}{2021}) \bibinfo{pages}{625--642}.
  \DOIprefix\doi{10.1007/s11401-021-0283-4}.
\bibitem[{Coutinho and Godsil(2021)}]{CG21b}
\bibinfo{author}{G.~Coutinho}, \bibinfo{author}{C.~D. Godsil},
  \bibinfo{title}{Graph spectra and continuous quantum walks},
  \bibinfo{publisher}{In preparation}, \bibinfo{year}{2021}.
\bibitem[{Godsil(1993)}]{Go93}
\bibinfo{author}{C.~D. Godsil}, \bibinfo{title}{Algebraic combinatorics},
  Chapman and Hall Mathematics Series, \bibinfo{publisher}{Chapman \& Hall, New
  York}, \bibinfo{year}{1993}.
\bibitem[{Godsil(2011)}]{Go08}
\bibinfo{author}{C.~D. Godsil},
\newblock \bibinfo{title}{Periodic graphs},
\newblock \bibinfo{journal}{\emph{Electron. J. Combin.}}   \textbf{18}~(1)
  (\bibinfo{year}{2011}) \bibinfo{pages}{\#P23}. \DOIprefix\doi{10.37236/510}.
\bibitem[{Godsil(2012)}]{Go12b}
\bibinfo{author}{C.~D. Godsil},
\newblock \bibinfo{title}{State transfer on graphs},
\newblock \bibinfo{journal}{\emph{Discrete Math.}}   \textbf{312}~(1)
  (\bibinfo{year}{2012}) \bibinfo{pages}{129--147}.
  \DOIprefix\doi{10.1016/j.disc.2011.06.032}.
\bibitem[{Godsil(2012)}]{Go12a}
\bibinfo{author}{C.~D. Godsil},
\newblock \bibinfo{title}{When can perfect state transfer occur?},
\newblock \bibinfo{journal}{\emph{Electron. J. Linear Algebra}}   \textbf{23}
  (\bibinfo{year}{2012}) \bibinfo{pages}{877--890}.
  \DOIprefix\doi{10.13001/1081-3810.1563}.
\bibitem[{Godsil and Lato(2020)}]{GL20}
\bibinfo{author}{C.~D. Godsil}, \bibinfo{author}{S.~Lato},
\newblock \bibinfo{title}{Perfect state transfer on oriented graphs},
\newblock \bibinfo{journal}{\emph{Linear Algebra Appl.}}   \textbf{604}
  (\bibinfo{year}{2020}) \bibinfo{pages}{278--292}.
  \DOIprefix\doi{10.1016/j.laa.2020.06.025}.
\bibitem[{Guo and Mohar(2017)}]{GM17}
\bibinfo{author}{K.~Guo}, \bibinfo{author}{B.~Mohar},
\newblock \bibinfo{title}{Hermitian adjacency matrix of digraphs and mixed
  graphs},
\newblock \bibinfo{journal}{\emph{J. Graph Theory}}   \textbf{85}~(1)
  (\bibinfo{year}{2017}) \bibinfo{pages}{217--248}.
  \DOIprefix\doi{10.1002/jgt.22057}.
\bibitem[{Hardy and Wright(2008)}]{HW08}
\bibinfo{author}{G.~H. Hardy}, \bibinfo{author}{E.~M. Wright},
  \bibinfo{title}{An introduction to the theory of numbers},
  \bibinfo{edition}{sixth} ed., \bibinfo{publisher}{Oxford University Press,
  Oxford}, \bibinfo{year}{2008}. \bibinfo{note}{Revised by D. R. Heath-Brown
  and J. H. Silverman, With a foreword by Andrew Wiles}.
\bibitem[{Kadyan and Bhattacharjya(2023)}]{MB21a}
\bibinfo{author}{M.~Kadyan}, \bibinfo{author}{B.~Bhattacharjya},
\newblock \bibinfo{title}{Integral mixed circulant graphs},
\newblock \bibinfo{journal}{\emph{Discrete Math.}}   \textbf{346}~(1)
  (\bibinfo{year}{2023}) \bibinfo{pages}{113142}.
  \DOIprefix\doi{10.1016/j.disc.2022.113142}.
\bibitem[{Liu and Li(2015)}]{LL15}
\bibinfo{author}{J.~Liu}, \bibinfo{author}{X.~Li},
\newblock \bibinfo{title}{Hermitian-adjacency matrices and {H}ermitian energies
  of mixed graphs},
\newblock \bibinfo{journal}{\emph{Linear Algebra Appl.}}   \textbf{466}
  (\bibinfo{year}{2015}) \bibinfo{pages}{182--207}.
  \DOIprefix\doi{10.1016/j.laa.2014.10.028}.
\bibitem[{Liu and Zhou(2022)}]{LZ18}
\bibinfo{author}{X.~Liu}, \bibinfo{author}{S.~Zhou},
\newblock \bibinfo{title}{Eigenvalues of {C}ayley {G}raphs},
\newblock \bibinfo{journal}{\emph{Electron. J. Combin.}}   \textbf{29}~(2)
  (\bibinfo{year}{2022}) \bibinfo{pages}{\#P2.9}.
  \DOIprefix\doi{10.37236/8569}.
\bibitem[{Petkovi\'{c} and Ba\v{s}i\'{c}(2011)}]{PB11}
\bibinfo{author}{M.~D. Petkovi\'{c}}, \bibinfo{author}{M.~Ba\v{s}i\'{c}},
\newblock \bibinfo{title}{Further results on the perfect state transfer in
  integral circulant graphs},
\newblock \bibinfo{journal}{\emph{Comput. Math. Appl.}}   \textbf{61}~(2)
  (\bibinfo{year}{2011}) \bibinfo{pages}{300--312}.
  \DOIprefix\doi{10.1016/j.camwa.2010.11.005}.
\bibitem[{Ramanujan(2000)}]{R1918}
\bibinfo{author}{S.~Ramanujan},
\newblock \bibinfo{title}{On certain trigonometrical sums and their
  applications in the theory of numbers [{T}rans. {C}ambridge {P}hilos. {S}oc.
  {\bf 22} (1918), no. 13, 259--276]},
\newblock in: \bibinfo{booktitle}{Collected papers of {S}rinivasa {R}amanujan},
  \bibinfo{publisher}{AMS Chelsea Publ., Providence, RI}, \bibinfo{year}{2000},
  pp. \bibinfo{pages}{179--199}. \DOIprefix\doi{10.1016/s0164-1212(00)00033-9}.
\bibitem[{Saxena et~al.(2007)Saxena, Severini, and Shparlinski}]{SSS07}
\bibinfo{author}{N.~Saxena}, \bibinfo{author}{S.~Severini},
  \bibinfo{author}{I.~E. Shparlinski},
\newblock \bibinfo{title}{Parameters of integral circulant graphs and periodic
  quantum dynamics},
\newblock \bibinfo{journal}{\emph{Int. J. Quantum Inf.}}   \textbf{05}~(03)
  (\bibinfo{year}{2007}) \bibinfo{pages}{417--430}.
  \DOIprefix\doi{10.1142/S0219749907002918}.
\bibitem[{So(2006)}]{So06}
\bibinfo{author}{W.~So},
\newblock \bibinfo{title}{Integral circulant graphs},
\newblock \bibinfo{journal}{\emph{Discrete Math.}}   \textbf{306}~(1)
  (\bibinfo{year}{2006}) \bibinfo{pages}{153--158}.
  \DOIprefix\doi{10.1016/j.disc.2005.11.006}.
\bibitem[{Song(tion)}]{S22}
\bibinfo{author}{X.-K. Song}, \bibinfo{title}{Quantum state transfer on
  integral oriented circulant graphs}, \bibinfo{year}{submitted for
  publication}. \href{http://arxiv.org/abs/2204.05026v2}{{\texttt
  arXiv:2204.05026v2}}.
\bibitem[{Tan et~al.(2019)Tan, Feng, and Cao}]{TFC19}
\bibinfo{author}{Y.-Y. Tan}, \bibinfo{author}{K.~Feng},
  \bibinfo{author}{X.~Cao},
\newblock \bibinfo{title}{Perfect state transfer on abelian {C}ayley graphs},
\newblock \bibinfo{journal}{\emph{Linear Algebra Appl.}}   \textbf{563}
  (\bibinfo{year}{2019}) \bibinfo{pages}{331--352}.
  \DOIprefix\doi{10.1016/j.laa.2018.11.011}.

\end{thebibliography}

\end{document}